\documentclass{amsart}

\usepackage[english]{babel} % for correct hyphenation
\usepackage[utf8]{inputenc} %to use non-ASCII symbols
\usepackage[T1]{fontenc}
\usepackage{lmodern}
\usepackage{mathtools} % to use coloneqq
\usepackage{thmtools} % to use declaretheorem
\usepackage{enumitem} 
\usepackage{amssymb}
\usepackage{hyperref}
\usepackage{todonotes}
\usepackage{stmaryrd} % to use longmapsfrom
\usepackage{cleveref} %to use cref, to be loaded last among similar packages

\declaretheorem[
name         = Theorem,
refname      = {Theorem, Theorems}, 
Refname      = {Theorem, Theorems},
numberwithin = section,
style        = plain,
]{theorem}

\declaretheorem[
name         = Proposition,
refname      = {Proposition, Propositions},
Refname      = {Proposition, Propositions},
sibling      = theorem,
style        = plain,
]{proposition}

\declaretheorem[
name         = Lemma,
refname      = {Lemma, Lemmas},
Refname      = {Lemma, Lemmas},
sibling      = theorem,
style        = plain,
]{lemma}

\declaretheorem[
name         = Claim,
refname      = {Claim,Claims},
Refname      = {Claim,Claims},
numberwithin      = theorem,
style        = plain,
]{claim}

\declaretheorem[
name 				= Definition,
refname 			= {Definition,Definitions},
Refname			= {Definition,Definitions},
sibling			= theorem,
style				= definition,
]{definition}

\declaretheorem[
name				= Example, 
refname			= {Example,Examples}, 
Refname			= {Example,Examples}, 
sibling			= theorem, 
style				= definition,
]{example}

\crefname{axiom}{Axiom}{Axioms}
\Crefname{axiom}{Axiom}{Axioms}

\newenvironment{claimproof}[1][Proof of Claim]{\begin{proof}[#1]}{\end{proof}}

\newcommand{\N}{\mathbb{N}}
\newcommand{\Z}{\mathbb{Z}}

\newcommand{\calMV}{\mathcal{MV}}
\newcommand{\bM}{\mathbf{M}}
\newcommand{\bA}{\mathbf{A}}
\newcommand{\bF}{\mathbf{F}}
\newcommand{\bC}{\mathbf{C}}

\newcommand{\MVM}{\mathsf{MVM}}

\newcommand{\ULM}{\mathsf{u}\ell\mathsf{M}}

\newcommand{\mvm}{MV-monoidal algebra}
\newcommand{\mvms}{MV-monoidal algebras}

\newcommand{\mv}{MV-algebra}
\newcommand{\mvs}{MV-algebras}

\newcommand{\ulm}{unital commutative distributive $\ell$-monoid}
\newcommand{\ulms}{unital commutative distributive $\ell$-monoids}

\newcommand{\ulg}{unital abelian $\ell$-group}
\newcommand{\ulgs}{unital abelian $\ell$-groups}

\newcommand{\df}{\coloneqq}

\renewcommand{\leq}{\leqslant} % nicer less or equal
\renewcommand{\geq}{\geqslant} % nicer greater or equal

\title{A finite axiomatization of positive MV-algebras}
\keywords{MV-algebras, positive subreducts, lattice-ordered monoids, quasivarieties, quasi-equations, axiomatization.}

\author[M. Abbadini]{Marco Abbadini}
\address{Department of Mathematics, Università degli Studi di Salerno,
	Fisciano (SA), Italy.}
\email{mabbadini@unisa.it}

\author[P. Jipsen]{Peter Jipsen}
\address{
	Keck Center of Science and Engineering,
	Faculty of Mathematics,
	Chapman University,
	Orange, USA.}
\email{jipsen@chapman.edu}

\author[T. Kroupa]{Tom\'a\v{s} Kroupa}
\address{
	Artificial Intelligence Center,
	Faculty of Electrical Engineering,
	Czech Technical University in Prague, Czech Republic.}
\email{tomas.kroupa@fel.cvut.cz}

\author[S. Vannucci]{Sara Vannucci$^{\ast}$}
\address{
	Artificial Intelligence Center,
	Faculty of Electrical Engineering,
	Czech Technical University in Prague, Czech Republic.}
\email{vanucsar@fel.cvut.cz}
\thanks{$^\ast$Corresponding author.}

\subjclass[2020]{Primary: 06D35. Secondary: 06F05, 08C15}

\begin{document}
	
	\begin{abstract}
		Positive MV-algebras are the subreducts of MV-algebras with respect to the signature $\{\oplus, \odot, \lor, \land, 0, 1\}$.
		We provide a finite quasi-equational axiomatization for the class of such algebras.
	\end{abstract}
	
	\maketitle

	\section{Introduction}
	
	An MV-algebra is an algebraic structure with a binary operation $\oplus$, a unary operation $\neg$ and a constant $0$ satisfying certain equational axioms. MV-algebras arose in the literature as the algebraic semantics of {\L}ukasiewicz logic and they are categorically equivalent to lattice-ordered abelian groups with strong unit (unital abelian $\ell$-groups, for short). Throughout the paper we assume familiarity with MV-algebras and unital abelian lattice-ordered groups; see \cite{CignoliOttavianoMundici00, Mundici11} and \cite{BigardKeimelWolfenstein77, Goodearl86} for background information.
	
	Boolean algebras are the algebraic semantics of classical propositional logic and they constitute a subvariety of MV-algebras;
	indeed, Boolean algebras are the MV-algebras satisfying the additional axiom $x \oplus x = x$.
	It is well known that bounded distributive lattices are precisely the algebras $\textbf{B} = \langle B, \vee, \wedge, 0, 1 \rangle$ that are isomorphic to a subreduct of some Boolean algebra.
	In this paper we investigate certain algebras, called \emph{positive MV-algebras}, which are to MV-algebras what bounded distributive lattices are to Boolean algebras.
	Positive MV-algebras are defined as the algebras $\mathbf{A} = \langle A, \oplus, \odot, \vee, \wedge, 0, 1 \rangle$ that are isomorphic to a subreduct of some MV-algebra.
	
	Whereas bounded distributive lattices are a variety axiomatized by finitely many equations, positive MV-algebras are a quasivariety but not a variety.
	Our main result is an axiomatization of positive MV-algebras via finitely many quasi-equations.
	The strategy is the following.
	Mundici's equivalence between MV-algebras and unital abelian $\ell$-groups was generalized by the first author to an equivalence between the categories of MV-monoidal algebras and unital commutative distributive $\ell$-monoids \cite{Abbadini2021-article}. 
	We observe that this equivalence restricts to an equivalence between positive MV-algebras and those unital commutative distributive $\ell$-monoids that are cancellative.
	We derive the axiomatization for positive MV-algebras by adding to the axioms of monoidal MV-algebras a quasi-equation expressing the cancellation property.
	
	Finally, in the appendix, we obtain an analogue of Booleanization for positive MV-algebras. In particular, we prove that the inclusion of a positive MV-algebra into an MV-algebra that is generated by the image of such inclusion is universal.
	
	The paper is organized as follows. In Section~2 we provide basic notions and examples regarding positive MV-algebras. In Section~3 we recall the definition of MV-monoidal algebras, their equivalence with {\ulms}, and we restrict this equivalence to positive MV-algebras. In Section~4 we obtain our main result: a finite quasi-equational axiomatization of positive MV-algebras.
	In the appendix we present the analogue of the free Boolean extension in the MV-algebraic setting.

	\section{Positive MV-algebras}
	
	\begin{definition} \label{d:MV}
		An \emph{MV-algebra} is an algebra $\langle M, \oplus, \neg, 0\rangle$ (arities $2,1,0$) with the following properties.
		\begin{enumerate}
			
			\item
			$\langle M, \oplus, 0 \rangle$ is a commutative monoid.
			
			\item
			$\neg \neg x = x$.
			
			\item
			$x \oplus \neg 0 = \neg 0$.
			
			\item
			$\neg (\neg x \oplus y) \oplus y = \neg (\neg y \oplus x) \oplus x$.
			
		\end{enumerate}
	\end{definition}
	
	A prime example is the so-called \emph{standard MV-algebra} $[0,1]$, with operations $x \oplus y \coloneqq \min\{x + y, 1\}$, $\lnot x \coloneqq 1 - x$, $0 \coloneqq 0$.
	
	It is usual to expand the signature of MV-algebras with the constant $1$ and the binary operations $\odot, \vee, \wedge$ defined as follows:
	$1 \coloneqq \neg 0$, $x \odot y \coloneqq \neg (\neg x \oplus \neg y)$, $x \vee y \coloneqq (x \odot \neg y) \oplus y$, $x \wedge y \coloneqq \neg(\neg x \vee \neg y)$. Let $\calMV$ be the variety of MV-algebras in the signature containing all the above operations. Using this enriched signature we can denote a member of $\calMV$ as 
	\[
	\bM = \langle M, \oplus, \odot, \vee, \wedge, \neg, 0, 1\rangle.
	\]
	Note that all the operations of $\bM$ except $\neg$ are \emph{positive} (i.e.\ order-preserving in each argument). 
	In this paper we focus on the \emph{positive subreducts of MV-algebras}, which are precisely the algebras
	\[
	\bA = \langle A, \oplus, \odot, \vee, \wedge, 0, 1 \rangle, 
	\]
	where $\bA$ is a subreduct of an MV-algebra $\bM$. 
	
	\begin{definition}
		An algebra of type $\{\oplus, \odot, \vee, \wedge, 0, 1\}$ is a \emph{positive MV-algebra} if it is isomorphic to a positive subreduct of some MV-algebra. We let $\mathcal{MV}_+$ denote the class of all positive MV-algebras.
	\end{definition}
	
	A prime example of a positive MV-algebra is the positive reduct of the standard MV-algebra $[0,1]$. We call this reduct the \emph{standard positive MV-algebra}.
	
	\begin{definition}
		A \em{McNaughton function} is a function $g\colon[0,1]^n \to [0,1]$
		which is continuous and piecewise linear with
		integer coefficients, i.e., there exist linear polynomials $p_1, \dots, p_t\colon[0, 1]^n \to [0,1]$ with integer coefficients, such that for any $(x_1, \dots, x_n) \in [0,1]^n$ there is $j$, $1 \leq j \leq t$ with $g(x_1, \dots, x_n) = p_j(x_1, \dots, x_n)$. 
	\end{definition}
	
	For each natural number $n$, the free $n$-generated MV-algebra is isomorphic to the algebra $\bF_n$ of McNaughton functions $[0,1]^n\to [0,1]$; see \cite{McNaughton51} or \cite[Chapter 3]{CignoliOttavianoMundici00}.

	\begin{example}[Non-decreasing McNaughton functions]\label{free}
		Let us equip $[0,1]$ with the standard order of reals, and let $\leq$ be the product order of $[0,1]^n$.  Then the algebra $\bF^{\leq}_n$ of non-decreasing (i.e.\ order-preserving) McNaughton functions is a positive subreduct of $\bF_n$.
	\end{example}
	
	We remark that McNaughton functions in $\bF^{\leq}_n$ correspond to positive formulae in {\L}ukasiewicz logic. This observation follows from an easy extension of the original McNaughton theorem; see \cite[Theorem 3.5]{CintulaKroupa13}.
	
	%In the following, we show that $\calP$ is the quasi-variety generated by the positive MV-algebra $[0,1]$, but it is not a variety.
	
	\begin{proposition} \label{p:generation}
		The class of positive MV-algebras is the quasivariety generated by the standard positive MV-algebra $[0,1]$.
	\end{proposition}
	\begin{proof}
		It is well known that if $\mathcal{Q}$ is a quasivariety of type $\Sigma$ generated by a class $\mathcal{K}$, and $\Sigma'$ is a subsignature of $\Sigma$, the class of isomorphic copies of subalgebras of $\Sigma'$-reducts of algebras in $\mathcal{Q}$ is the quasivariety generated by the $\Sigma'$-reducts of algebras in $\mathcal{K}$. 
		By Di Nola's representation theorem (see \cite[Theorem~9.5.1]{CignoliOttavianoMundici00}), the standard MV-algebra $[0,1]$ generates the class of MV-algebras as a quasivariety.
		The desired result follows.
	\end{proof}
	
	Next, we show that the quasivariety $\mathcal{MV}_+$ is not a variety.
	To this purpose, we recall the definition of Chang's MV-algebra.
	
	\begin{definition}
		The \emph{Chang's MV-algebra} is a linearly ordered MV-algebra---denoted in this paper by $\bC$---defined on the set of formal symbols
		\[C \coloneqq \{0, \varepsilon, 2\varepsilon, 3\varepsilon, \dots\} \cup \{\dots, 1-3\varepsilon, 1 - 2\varepsilon, 1-\varepsilon, 1\}
		\]
		(where we think of $\varepsilon$ as an infinitesimal) and with the following operations:
		\[
		x \oplus y \coloneqq \begin{cases}
			(m+n)\varepsilon, & \text{if $x=n\varepsilon$, $y=m\varepsilon$, } \\
			1-(m-n)\varepsilon, & \text{if $x=1-m\varepsilon$, $y=n\varepsilon$ and $0<n<m$,} \\
			1-(n-m)\varepsilon, & \text{if $x=m\varepsilon$, $y=1-n\varepsilon$ and $0<m<n$,} \\
			1 & \text{otherwise,}
		\end{cases}
		\]
		and
		\[
		\neg x\coloneqq \begin{cases}
			1-n\varepsilon, & \text{if $x=n\varepsilon$, } \\
			n\varepsilon, & \text{if $x=1-n\varepsilon$.}
		\end{cases}
		\]
		
	\end{definition}
	\begin{example}
		Let $\bC$ be Chang's MV-algebra. The algebra $\bC'$ with the universe $\{0,\varepsilon,2 \varepsilon,\dots,1\}$ is a positive subreduct of $\bC$.
		Let $\theta$ be an equivalence relation on the algebra $\bC'$ with classes $\{0\}$, $\{\varepsilon,2\varepsilon,\dots\}$, and $\{1\}$. Then $\theta$ is a $\mathcal{MV}_+$-congruence on $\bC'$. The quotient $\bC'/\theta$ is isomorphic to the three-element algebra $\{\bar{0},\bar{\varepsilon},\bar{1}\}$ that satisfies $\bar{\varepsilon} \oplus \bar{\varepsilon} = \bar{\varepsilon} $ and $\bar{\varepsilon}\odot \bar{\varepsilon} = \bar{0}$.
		The quasi-equation
		\[
		x \oplus x = x \Rightarrow x \odot x = x
		\]
		holds for all MV-algebras, and thus for all positive MV-algebras, as well. However, this quasi-equation does not hold in $\bC'/\theta$, as witnessed by the element $\bar{\varepsilon}$.
		Therefore, $\bC'/\theta$ is not a positive MV-algebra.
	\end{example}
	
	\begin{theorem}[{\cite[Theorem~3.5]{CintulaKroupa13}}] \label{l:CintulaKroupa}
		For every $n \in \N$, the subalgebra of the power $[0,1]^{[0,1]^n}$ of the standard positive MV-algebra $[0,1]$ generated by the projections consists precisely of the non-decreasing (i.e.\ order-preserving) McNaughton functions from $[0,1]^n$ to $[0,1]$.
	\end{theorem}
	
	\begin{proposition}
		The free $n$-generated positive MV-algebra is isomorphic to the positive subreduct $\bF^{\leq}_n$ of $\bF_n$ from \cref{free}.
	\end{proposition}
	\begin{proof}
		By \cref{p:generation}, $\mathcal{MV}_+$ is the quasivariety generated by the standard positive MV-algebra $[0,1]$; 
		thus, the free $n$-generated positive algebra is (up to isomorphism) the subalgebra of the power $[0,1]^{[0,1]^n}$ of the standard positive MV-algebra $[0,1]$ generated by the projections, which, by \cref{{l:CintulaKroupa}}, is $\bF^{\leq}_n$.
	\end{proof}

	\section{MV-monoidal algebras}
	
	Every positive MV-algebra satisfies the equations (\ref{ax:A1}--\ref{ax:A7}) in Definition \ref{d:MVM} below (see \cite{CignoliOttavianoMundici00}). Since the algebras thus defined are of separate interest \cite{Abbadini2021-article}, we recall their definition.
	
	\begin{definition} \label{d:MVM}
		An \emph{{\mvm}} is an algebra $\langle A, \oplus, \odot, \lor, \land, 0, 1 \rangle$ (arities $2$, $2$, $2$, $2$, $0$, $0$) satisfying the following equational axioms. 
		\begin{enumerate}[label = (E\arabic*), ref = E\arabic*]
			
			\item	\label{ax:A1}
			$\langle A, \lor, \land \rangle$ is a distributive lattice.		
			
			\item	\label{ax:A2}
			$\langle A, \oplus, 0 \rangle$ and $\langle A, \odot, 1 \rangle$ are commutative monoids.
			
			\item	\label{ax:A3}
			The operations $\oplus$ and $\odot$ distribute over both $\lor$ and $\land$.		
			
			\item	\label{ax:A4}
			$(x \oplus y) \odot ((x \odot y) \oplus z) = (x \odot (y \oplus z)) \oplus (y \odot z)$.
			
			\item \label{ax:A5}
			$(x \odot y) \oplus ((x \oplus y) \odot z) = (x \oplus (y \odot z)) \odot (y \oplus z)$.
			
			\item	\label{ax:A6}
			$(x \odot y) \oplus z = ((x \oplus y) \odot ((x \odot y) \oplus z)) \lor z$.
			
			\item	\label{ax:A7}
			$(x \oplus y) \odot z = ((x \odot y) \oplus ((x \oplus y) \odot z)) \land z$.
			
		\end{enumerate}
	\end{definition}
	
	The main use of \eqref{ax:A2}, \eqref{ax:A4}, \eqref{ax:A5} is to prove the associativity of the monoid operation $+$ in the enveloping {\ulm} of $\mathbf{A}$.
	The main use of \eqref{ax:A3}, \eqref{ax:A6} (resp.\ \eqref{ax:A3},\eqref{ax:A7}) is to prove that $+$ distributes over $\lor$ (resp.\ $\land$).
	
	We denote with $\MVM$ the category of {\mvms} with homomorphisms. 
	
	\begin{definition} \label{d:l-group}
		An \emph{abelian lattice-ordered group} (\emph{abelian $\ell$-group}, for short) is an algebra $\langle G, +, \lor, \land, 0, - \rangle$ (arities 2,2,2,0,1) such that $\langle G, +, 0, - \rangle$ is an abelian $\ell$-group, $\langle G, \lor, \land \rangle$ is a lattice, and $+$ distributes over $\lor$ and $\land$.
		An element $1$ of an abelian $\ell$-group $\mathbf{G}$ is a \emph{strong unit} (or simply a unit) if $1 \geq 0$ and, for every $g \in G$, there is $n \in \N$ such that $g \leq n1$.
		In this paper, by a \emph{unital abelian lattice-ordered group} we mean an algebra $\langle G, +, \lor, \land, 0, 1, -1, - \rangle$ such that $\langle G, +, \lor, \land, 0, - \rangle$ is an abelian $\ell$-group, $1$ is a unit, and $-1$ is the group-inverse of $1$.
	\end{definition}
	
	We will further consider the category of MV-algebras (the category whose objects are MV-algebras and whose morphisms are MV-homomorphisms) and the category of unital abelian $\ell$-groups (the category whose objects are unital abelian $\ell$-groups and morphisms are group- and lattice-homomorphisms preserving units). These two categories are equivalent by a theorem of Mundici \cite{Mundici86}. These results can be lifted to an equivalence between unital commutative distributive lattice-ordered monoids and MV-monoidal algebras; see \cite{Abbadini2021-article} or \cite[Chapter~4]{Abbadini2021-thesis} for details. Since our proof of quasi-equational characterization of positive MV-algebras is essentially based on this equivalence, we provide a summary and some consequences of this result.

	\begin{definition} \label{d:ulm}
		A \emph{commutative distributive lattice-ordered monoid with strong units} (henceforth shortened to \emph{\ulm}) is an algebra $\langle M, +, \lor, \land, 0, 1, -1 \rangle$ (arities $2,2,2,0,0,0$) with the following properties.
		\begin{enumerate}[label = (M\arabic*),  ref = M\arabic*,  start = 1]
			
			\item \label[axiom]{ax:M1-gen}
			$\langle M, \lor, \land \rangle$ is a distributive lattice.
			
			\item	\label[axiom]{ax:M2-gen} 
			$\langle M, +, 0 \rangle$ is a commutative monoid.
			
			\item \label[axiom]{ax:M3-gen} 
			The operation $+$ distributes over $\lor$ and $\land$ on both sides.
			
			\item	\label[axiom]{ax:U1}
			$-1 + 1 = 0$.
			
			\item \label[axiom]{ax:U2}
			$-1 \leq 0 \leq 1$.
			
			\item \label[axiom]{ax:U3}
			For all $x \in M$, there exists $n \in \N$ such that 
			\[
			\underbrace{(-1) + \dots + (-1)}_{n\ \text{times}}\leq x\leq \underbrace{1 + \dots + 1}_{n\ \text{times}}.
			\]
			
		\end{enumerate}
	\end{definition}
	For $n \in \N$, we write $n$ for $1 + \dots + 1$ ($n$ times) and $-n$ for $(-1) + \dots + (-1)$ ($n$ times), and, we write $x - n$ for $x + (-n)$. 
	For {\ulgs}, the constant $-1$ is term-definable via the constant $1$ and the group inverse operation.
	Since in the context of {\ulms} we do not have the group-inverse operation, we stress the fact that $-1$ is an indivisible formal symbol, i.e.\ a constant symbol in the signature.
	
	We let $\ULM$ denote the category of {\ulms} with homomorphisms. For a {\ulm} $\mathbf{M}$, we set
	\[
	\Gamma(\mathbf{M})\df \{x\in M\mid 0\leq x\leq 1 \}.
	\]
	We equip $\Gamma(\mathbf{M})$ with the operations of {\mvms}: we define $\lor$, $\land$, $0$ and $1$ by the restriction, and we set
	\[
	x \oplus y \df (x + y) \land 1,  \qquad  x \odot y \df (x + y - 1) \lor 0.
	\]
	The algebra $\langle \Gamma(\mathbf{M}), \oplus, \odot, \lor, \land, 0, 1 \rangle$ is an {\mvm} (see \cite[Proposition~3.6]{Abbadini2021-article} or \cite[Theorem~4.29]{Abbadini2021-thesis}).
	Given a morphism of {\ulms} $f\colon \mathbf{M}\to \mathbf{N}$, we let $\Gamma(f)$ denote its restriction $\Gamma(f)\colon \Gamma(\mathbf{M})\to \Gamma(\mathbf{N})$. This yields a functor
	\[
	\Gamma\colon \ULM\to \MVM.
	\]
	
	\begin{theorem}[{\cite[Theorem~8.21]{Abbadini2021-article} or \cite[Theorem~4.74]{Abbadini2021-thesis}}]\label{t:Gam-is-equivalence}
		The functor 
		\[
		\Gamma \colon \ULM \to \MVM
		\]
		determines an equivalence of categories.
	\end{theorem}
	\noindent
	A quasi-inverse for $\Gamma$ is described in \cite[Chapter~4, Section~6]{Abbadini2021-thesis} and is denoted by $\Xi$; we mention that this construction---based on the notion of good $\Z$-sequence---is suggested by the following result.
	
	\begin{proposition}[{\cite[Theorem~4.72]{Abbadini2021-thesis}}] \label{p:good}
		Let $\mathbf{M}$ be a {\ulm}, and let $x, y \in M$.
		If, for every $n \in \Z$, we have
		\[
		((x-n) \lor 0) \land 1 = ((y-n) \lor 0) \land 1,
		\]
		then $x = y$.
	\end{proposition}
	
	\Cref{p:good} is a consequence of the following fact.
	
	\begin{lemma}[{\cite[Proposition~4.68]{Abbadini2021-thesis}}]\label{l:sum-slices}
		For every {\ulm} $\mathbf{M}$, every $x \in M$, and all $n \leq m \in \Z$ such that $-n \leq x \leq m$, we have
		\[
		x = n + \sum_{i = n}^{m} ((x - n) \lor 0) \land 1.
		\]
	\end{lemma}

	\begin{proposition} \label{p:preservation}
		The functors $\Gamma$ and $\Xi$ preserve and reflect injectivity of morphisms.
	\end{proposition}
	
	\begin{proof}		
		It is enough to show that $\Gamma$ preserves and reflects injectivity.		
		Since $\Gamma$ is defined as the restriction on morphisms, $\Gamma$ preserves injectivity.		
		Let us prove that $\Gamma$ reflects injectivity.
		Let $f \colon \mathbf{M} \to \mathbf{N}$ be a morphism of {\ulms}, and suppose $\Gamma(f)$ to be injective.
		Let $x, y \in \mathbf{M}$ and suppose $f(x) = f(y)$.
		For every $n \in \Z$, we have
		\begin{align*}
			\Gamma(f)(((x - n) \lor 0) \land 1) & = f(((x - n) \lor 0) \land 1)\\
			& = ((f(x) - n) \lor 0) \land 1\\
			& = ((f(y) - n) \lor 0) \land 1\\
			& = f(((y - n) \lor 0) \land 1)\\                                    
			& = \Gamma(f)(((y - n) \lor 0) \land 1).
		\end{align*}
		Since $\Gamma(f)$ is injective, we deduce
		\[
		((x - n) \lor 0) \land 1 = ((y - n) \lor 0) \land 1.
		\]
		Thus, $f$ is injective.
		Since this holds for every $n \in \Z$, by \cref{p:good} we have $x = y$.
		This proves that $\Gamma$ reflects injectivity.
	\end{proof}
	
	\begin{lemma} \label{l:correspondence}
		The $\{+, \lor, \land, 0, 1, -1\}$-reduct of a {\ulg} is a {\ulm}, and the $\{\oplus, \odot, \lor, \land, 0, 1\}$-reduct of an {\mv} is an {\mvm}.
		Moreover, 
		\begin{enumerate}
			\item \label{i:corr-reducts}
			the equivalence $\Gamma \colon \ULM \longleftrightarrow \MVM \colon \Xi$ restricts to an equivalence between isomorphic copies of reducts of {\ulgs} and isomorphic copies of reducts of {\mvs}, and
			
			\item \label{i:corr-subreducts}
			the equivalence $\Gamma \colon \ULM \longleftrightarrow \MVM \colon \Xi$ restricts to an equivalence between isomorphic copies of subreducts of {\ulgs} and isomorphic copies of subreducts of {\mvs}.
		\end{enumerate}
	\end{lemma}
	
	\begin{proof}
		Up to \eqref{i:corr-reducts} is shown in \cite[Appendix~A]{Abbadini2021-article} (see \cite[Chapter~4, Section~8]{Abbadini2021-thesis} for a more detailed proof).
		\eqref{i:corr-subreducts} follows from \eqref{i:corr-reducts}, together with the fact that $\Gamma$ and $\Xi$ preserve injectivity of morphisms (\cref{p:preservation}).
	\end{proof}

	\section{Main Result}
	
	We recall the following folklore result.
	
	\begin{lemma} \label{l:g-generation}
		Let $M$ be a generating subset of an abelian group $\mathbf{G}$, and suppose that $M$ is closed under $+$ and $0$.
		For every $z \in G$, there are $x, y \in M$ such that $z = x - y$.
	\end{lemma}
	
	\Cref{l:g-generation} generalizes in the setting of (unital) abelian $\ell$-groups, as follows.
	
	\begin{lemma}\label{l:lg-generation}
		Let $M$ be a generating subset of an abelian $\ell$-group (resp.\ unital abelian $\ell$-group) $\mathbf{G}$, and suppose that $M$ is closed under $+$, $\lor$, $\land$ and $0$ (resp.\ $+$, $\lor$, $\land$, $0$, $1$ and $-1$).
		For every $z \in G$, there are $x, y \in M$ such that $z = x - y$.
	\end{lemma}
	
	\begin{proof}
		It is enough to prove that the set $\{ x - y \mid x, y \in M \}$
		contains $M$ and is closed under $+$, $\lor$, $\land$, $0$ and $-$ (resp.\ $+$, $\lor$, $\land$, $0$, $1$, $-1$ and $-$).
		These verifications are easy; we only show closure under $\lor$:
		\[
		(x - y) \lor (x' - y') = ((x + y') \lor (x' + y)) - (y + y'). \qedhere
		\]
	\end{proof}
	
	Recall that a monoid $\mathbf{M}$ is called \emph{cancellative} if, for all $x, y, z \in M$, the condition $x + z = y + z$ implies $x = y$.
	
	\begin{proposition} \label{p:cancellative-monoid}
		An algebra $\langle M, +, \lor, \land, 0 \rangle$ is isomorphic to a subreduct of an abelian $\ell$-group if and only if it is a cancellative commutative distributive $\ell$-monoid.
	\end{proposition}
	
	\begin{proof}
		The left-to-right implication follows from the fact that the conditions that define cancellative commutative distributive $\ell$-monoids are quasi-equations that hold in all abelian $\ell$-groups.
		%Since quasi-equations are preserved under isomorphisms and subalgebras, every positive subreduct of an abelian $\ell$-group is a cancellative commutative distributive $\ell$-monoid.
		
		For the converse implication, let $\mathbf{M}$ be a cancellative commutative distributive $\ell$-monoid.
		Following a standard procedure (see \cite{Steinitz}) motivated by \cref{l:g-generation}, one embeds $\mathbf{M}$ into its so-called Grothendieck group (\cite{Grothendieck}), or algebra of fractions.
		The elements of $G$ are the equivalence classes of the relation $\sim$ on $M \times M$ defined by
		\[
		(a,b) \sim (c,d) \iff a + d = c + b.
		\]
		For $a, b \in M$, we write $[a,b]$ for the equivalence class of $(a,b)$ with respect to $\sim$.
		The group operations on $G$ are defined as follows:
		\begin{align*}
			[a,b] + [c,d] & \df[a + c, b + d];\\
			0 & \df [0,0];\\
			-[a,b]& \df [b,a].
		\end{align*}
		One then defines the following additional operations on $G$:
		\begin{align*}
			[a,b] \lor [c,d] & \df [(a+d) \lor (c+b), b+d];\\
			[a,b] \land [c,d] & \df [(a+d) \land (c+b),b+d].
		\end{align*}
		The algebra $\mathbf{G} = \langle G, +, \lor, \land, 0, - \rangle$ is an abelian $\ell$-group, and the map
		\begin{align*}
			\iota \colon M & \longrightarrow G\\
			x & \longmapsto [x,0]
		\end{align*}
		is an injective homomorphism from $\langle M, +, \lor, \land, 0 \rangle$ to $\langle G, +, \lor, \land, 0 \rangle$.
		This shows that $\mathbf{M}$ is isomorphic to a subreduct of $\mathbf{G}$.
	\end{proof}

	\begin{proposition} \label{p:cancellative-unital-monoid}
		An algebra $\langle M, +, \lor, \land, 0, 1, -1 \rangle$ is isomorphic to a subreduct of a unital abelian $\ell$-group if and only if it is a cancellative {\ulm}.
	\end{proposition}
	
	\begin{proof}
		Any subreduct of a unital abelian $\ell$-group is a cancellative {\ulm}: \eqref{ax:U3} holds because, for any $n \in \N \setminus \{0\}$, the validity of an equation $-n \leq x \leq n$ is preserved by subalgebras.
		All the remaining axioms are quasi-equations, so their validity is preserved.
		
		For the converse implication, given a cancellative {\ulm} $\langle M, +, \lor, \land, 0 \rangle$, we consider the {\ulg} $\langle G, +, \lor, \land, 0, - \rangle$ built from $\langle M, +, \lor, \land, 0 \rangle$ as in the proof of \cref{p:cancellative-monoid}.
		We show that the element $[1,0]$ is a unit of the abelian $\ell$-group $\langle G, +, \lor, \land, 0, - \rangle$.
		We have $[0,0] \lor [1,0] = [1,0]$, and thus $[0,0] \leq [1,0]$.
		Moreover, given $a,b \in M$, let $n \in \N \setminus \{0\}$ be large enough so that $a \leq n$ and $b \geq -n$.
		Then $a - n \leq 0 \leq b + n$, and thus
		\[
		[a,b] \land [2n,0] = [a,b] \land [n,-n] = [(a -n) \land (b + n), b - n] = [a - n, b -n] = [a,b].
		\]
		Therefore, $[a,b] \leq [2n,0] = 2n[1,0]$.
		This shows that $[1,0]$ is a unit.
		Denoting the element $[1,0]$ with $1$ and its group-inverse $[-1,0]$ with $-1$, we conclude that the algebra $\langle G, +, \lor, \land, 0, 1, - 1, -\rangle$ is a {\ulg}.
		The function $\iota \colon M \hookrightarrow G$ from the proof of \cref{p:cancellative-unital-monoid} clearly preserves also the constants $1$ and $-1$.
	\end{proof}
	
	\begin{proposition}\label{p:through-gamma}
		The following are equivalent for a {\ulm} $\mathbf{M}$.
		\begin{enumerate}
			
			\item \label{i:cancellative}
			$\mathbf{M}$ is cancellative.
			
			\item \label{i:cancellative-1}
			For all $x, y, z \in \{w \in M \mid 0 \leq w \leq 1\}$, if $x + z = y + z$ then $x = y$.
			
			\item \label{i:cancellative-mv}
			For all $x, y, z \in \Gamma(\mathbf{M})$, if $x\oplus z=y\oplus z$ and $x\odot z=y\odot z$ then $x=y$.
			
		\end{enumerate}
	\end{proposition}
	
	\begin{proof}
		The implication \eqref{i:cancellative} $\Rightarrow$ \eqref{i:cancellative-1} is trivial, and the implication \eqref{i:cancellative-1} $\Rightarrow$ \eqref{i:cancellative-mv} is easy.
		
		Let us suppose \eqref{i:cancellative-mv}, and let us prove \eqref{i:cancellative-1}.
		Let $x,y,z \in \{w \in M \mid 0 \leq w \leq 1\}$ be such that $x + z = y + z$.
		We claim that, for every $n \in \Z$, we have
		\begin{equation} \label{eq:sum}
			((x + z - n) \lor 0) \land 1 = ((y + z - n) \lor 0) \land 1. \tag{\dag}
		\end{equation}
		For $n = 0$, resp.\ $n = 1$, \eqref{eq:sum} boils down to $x \oplus z = y \oplus z$, resp.\ $x \odot z = y \odot z$; both of these conditions are true by hypothesis.
		Since $x$, $y$ and $z$ are below $1$, both sides of \eqref{eq:sum} equal $0$ for $n \geq 2$.
		Since $x$, $y$ and $z$ are above $0$, both sides of \eqref{eq:sum} equal $1$ for $n \leq -1$.
		This settles the claim.
		By \cref{p:good}, we deduce $x = y$.
		
		We are left to prove the implication \eqref{i:cancellative-1} $\Rightarrow$ \eqref{i:cancellative}.
		
		Let us assume \eqref{i:cancellative-1}.
		Let us start by proving the following claim, where we have added the hypothesis `$0 \leq z \leq 1$'; we will then weaken this hypothesis to `$0 \leq z$', and we will finally prove the general case.
		
		\begin{claim} \label{cl:z-unital}
			For all $x, y, z \in M$ with $0 \leq z \leq 1$, if $x + z = y + z$ then $x = y$.
		\end{claim}
		
		\begin{claimproof}
			Let $x, y, z \in M$ with $0 \leq z \leq 1$, and suppose $x + z = y + z$.
			Since $x + z = y + z$, we have, for every $n \in \Z$,
			\begin{equation} \label{e:fundamental}
				((- n + x + z) \vee z) \wedge (1 + z) = ((-n + y + z) \vee z) \wedge (1 + z).
			\end{equation}
			By distributivity of $+$ over $\vee$ and $\wedge$, \cref{e:fundamental} can be written as
			\begin{equation} \label{e:fundamental-rewritten}
				(((- n + x) \lor 0) \land 1) + z = (((- n + y) \lor 0) \land 1) + z.
			\end{equation}
			From \cref{e:fundamental-rewritten}, using \eqref{i:cancellative-1}, we deduce
			\begin{equation} \label{e:strip}
				((- n + x) \lor 0) \land 1= ((- n + y) \lor 0) \land 1.
			\end{equation}
			Since \cref{e:strip} holds for every $n \in \Z$, by \cref{p:good} we have $x = y$.
		\end{claimproof}
		
		\begin{claim} \label{cl:z-positive}
			For all $x, y, z \in M$ with $z \geq 0$, if $x + z = y + z$ then $x = y$.
		\end{claim}
		
		\begin{claimproof}
			We prove this inductively on $n \in \N \setminus\{0\}$ such that $z \leq n$.
			The case $n = 1$ is \cref{cl:z-unital}.
			Let $n \in \N \setminus\{0,1\}$, and let us suppose that the statement holds for $n - 1$.
			Let $x, y, z \in M$ with $0 \leq z \leq n$, and suppose $x + z = y + z$.
			Set $w \df z \wedge (n - 1)$, and $v \df (z - (n-1)) \vee 0$.
			Then
			\begin{align*}
				w + v & = (z \wedge (n - 1)) + ((z - (n-1)) \vee 0)\\
				& = (z \wedge (n - 1)) + (z \vee (n - 1)) - (n-1)\\
				& = z + (n - 1) - (n-1)\\
				& = z,
			\end{align*}
			and thus the equality $x + z = y + z$ can be written as
			\[
			x + w + v = y + w + v.
			\]
			By \cref{cl:z-unital}, using the fact that $0 \leq v \leq 1$, we deduce
			\[
			x + w = y + w.
			\]
			By inductive hypothesis, using the fact that $0 \leq w \leq n-1$, we conclude $x = y$.
		\end{claimproof}
		
		Let $x, y, z \in M$ and suppose $x + z = y + z$.
		The element $z$ can be written as $u + n$, where $n \in \Z$ and $u \geq 0$.
		Then the equality $x + z = y + z$ can be written as $x + u + n = y + u + n$, which is equivalent to $x + u = y + u$.
		By \cref{cl:z-positive}, we deduce $x = y$.
		This proves \eqref{i:cancellative}: $\mathbf{M}$ is cancellative.
	\end{proof}
	
	We arrive at our main result: a finite axiomatization of the class of positive MV-algebras.
	
	\begin{theorem} \label{t:axiomatisation}
		The positive MV-algebras are precisely the {\mvms} that satisfy, for all $x$, $y$, and $z$,
		\begin{center}
			if $x\oplus z=y\oplus z$ and $x\odot z=y\odot z$, then $x=y$.
		\end{center}
	\end{theorem}
	
	\begin{proof}
		For the left-to-right implication, we note that the $[0,1]$ satisfies the equations defining MV-monoidal algebras and the quasi-equation in the statement.
		Therefore, by Proposition~\ref{p:generation}, the same is true for each positive MV-algebra.

		Let us prove the converse implication.
		Suppose $\mathbf{A} = \langle A, \oplus, \odot, \lor, \land, 0, 1 \rangle$ is an {\mvm}, and suppose that, for all $x, y, z \in A$, if $x \oplus z = y \oplus z$ and $x \odot z = y \odot z$, then $x = y$.
		Using the implication \eqref{i:cancellative-mv} $\Rightarrow$ \eqref{i:cancellative} in \cref{p:through-gamma}, we deduce that $\Xi(\mathbf{A})$ is cancellative.
		By \cref{p:cancellative-unital-monoid}, $\Xi(\mathbf{A})$ is isomorphic to a subreduct of a {\ulg}.
		By \cref{l:correspondence}, $\mathbf{A}$ is isomorphic to a subreduct of an {\mv}, i.e.\ $\mathbf{A}$ is a positive MV-algebra.
	\end{proof}
	
	\Cref{t:axiomatisation} provides a finite quasi-equational axiomatization for the class of positive MV-algebras consisting of the equational axioms defining MV-monoidal algebras (\ref{ax:A1}--\ref{ax:A7}) together with the quasi-equation
	\begin{center}
		if $x\oplus z=y\oplus z$ and $x\odot z=y\odot z$, then $x=y$,
	\end{center}
	which can be seen as an appropriate version of the cancellation property.

	\appendix
	
	\section{Free MV-extensions}

	By standard results in general algebra, the forgetful functor $U$ from MV-algebras to positive MV-algebras has a left adjoint $F$.
	%As an immediate consequence of the definition of positive MV-algebras and of the universal property of the unit of an adjunction, 
	For every positive MV-algebra $\mathbf{A}$, it is immediate that the component $\iota_{\mathbf{A}} \colon \mathbf{A} \to UF(\mathbf{A})$ at $\mathbf{A}$ of the unit is injective, and that the image of $\iota_{\mathbf{A}}$ generates the MV-algebra $F(\mathbf{A})$.
	We speak of the pair $(F(\mathbf{A}), \iota_{\mathbf{A}})$ (or simply of $F(\mathbf{A})$, leaving $\iota_{\mathbf{A}}$ understood) as the \emph{free MV-extension} of $\mathbf{A}$.
	In this section we prove that, given a positive subreduct $\mathbf{A}$ of an MV-algebra $\mathbf{B}$ such that $A$ generates $\mathbf{B}$, and denoting with $i$ the inclusion of $A$ into $B$, $(\mathbf{B}, i)$ is the free MV-extension of $\mathbf{A}$; moreover, under the same conditions, for every $x \in B$, there are $a_1, \dots, a_n, b_1, \dots, b_n \in A$ such that $x = \bigoplus_{j = 1}^{n} (a_j \odot \lnot b_j)$.
	
	The following Lemma corresponds, essentially, to the fact that Mundici's bijection between $\mathbf{G}$ and $\Xi(\Gamma(\mathbf{G}))$ preserves $+$.
	We refer to \cite[Proposition~4.67]{Abbadini2021-thesis} for a proof in the setting of {\ulms}.
	
	\begin{lemma} \label{l:sum-iso}
		Let $\mathbf{G}$ be a {\ulg}, and let $x, y \in G$.
		For each $n \in \Z$, set $x_n \coloneqq ((x -n) \lor 0) \land 1)$ and $y_n \coloneqq ((x -n) \lor 0) \land 1)$.
		We have
		\[
		((x + y) \lor 0) \land 1 = \bigoplus_{n \in \Z} (x_n \odot y_{-n-1}) = \bigodot_{n \in \Z} (x_n \oplus y_{-n}).
		\]
	\end{lemma}
	\noindent (The expression $\bigoplus_{n \in \Z} (x_n \odot y_{-n-1})$ makes sense because for all but finitely many $n \in \Z$ we have $x_n \odot y_{-n-1} = 0$.
	Analogously, the expression $\bigodot_{n \in \Z} (x_n \oplus y_{-n})$ makes sense because for all but finitely many $n \in \Z$ we have $x_n \oplus y_{-n} = 1$.)
	
	\begin{lemma} \label{l:sub}
		Let $\mathbf{G}$ be a unital abelian $\ell$-group.
		Set
		\[
		\mathcal{M} \coloneqq \{ M \subseteq G \mid M \text{ is closed under }+, \lor, \land, 0, 1, -1 \}
		\]
		and
		\[
		\mathcal{A} \coloneqq \{ A \subseteq \Gamma(\mathbf{G}) \mid A \text{ is closed under } \oplus, \odot, \lor, \land, 0, 1 \}.
		\]
		\begin{enumerate}
			\item \label{i:bijection} The sets $\mathcal{M}$ and $\mathcal{A}$ are in bijection, as witnessed by the functions
			\begin{align*}
				\mathcal{M} & \longleftrightarrow \mathcal{A}\\
				M &\longmapsto \{ x \in M \mid 0 \leq x \leq 1 \}\\
				\{ x \in G \mid \forall n \in \Z\  ((x - n) \lor 0) \land 1 \in A \} & \longmapsfrom A.
			\end{align*}
			\item \label{i:restriction}
			The bijection in \eqref{i:bijection} restricts to the subsets $\{M \in \mathcal{M} \mid M \text{ is closed under }-\}$ and $\{A \in \mathcal{A} \mid A \text{ is closed under }\lnot\}$.
			\item \label{i:generate}
			For every $M \in \mathcal{M}$, $M$ generates the unital abelian $\ell$-group $\mathbf{G}$ if and only if $\{ x \in M \mid 0 \leq x \leq 1 \}$ generates the MV-algebra $\Gamma(\mathbf{G})$.
		\end{enumerate}
	\end{lemma}
	
	\begin{proof}
		The fact that the function $f \colon \mathcal{M} \to \mathcal{A}$ is well-defined is immediate.
		
		Let us prove that the function $g \colon \mathcal{A} \to \mathcal{M}$ is well-defined.
		Let $A \in \mathcal{A}$, and set
		\[
		M \coloneqq \{ x \in G \mid \forall n \in \Z\  ((x - n) \lor 0) \land 1 \in A \}.
		\]
		The set $M$ is closed under $+$ by \cref{l:sum-iso}.
		Moreover, it is closed under $\lor$ because, for all $x, y \in M$, we have $((x \lor y) \lor 0) \land 1 = ((x  \lor 0) \land 1) \lor ((x  \lor 0) \land 1)$.
		Analogously, it is closed under $\land$.
		It is not difficult to prove that $0, 1, -1 \in M$.
		Therefore, $M \in \mathcal{M}$, and thus $g$ is well-defined.
		
		It is easy to see that the composite $f \circ g \colon \mathcal{A} \to \mathcal{A}$ is the identity on $\mathcal{A}$.
		To prove that $g \circ f$ is the identity on $\mathcal{M}$, let $M \in \mathcal{M}$.
		We shall prove that
		\[
		M = \{x \in G \mid \forall n \in \Z\ (x - n) \lor 0) \land 1 \in M\}.
		\]
		The left-to-right inclusion is immediate.
		For the converse inclusion, let $x \in G$ be such that, for every $n \in \Z$, $((x - n) \lor 0) \land 1 \in M$.
		Let $n \in \N$ be such that $-n \leq x \leq n$.
		Then (cf.\ \cref{l:sum-slices})
		\[
		x = -n + \sum_{i = -n}^{n-1} (((x - i) \lor 0) \land 1),
		\]
		and thus $x \in M$.
		This proves \eqref{i:bijection}.
		
		If $M \in \mathcal{M}$ is closed under $-$, $f(M)$ is easily seen to be closed under $\lnot$.
		If $A \in \mathcal{A}$ is closed under $\lnot$, $g(A)$ is closed under $-$ because, for every $x \in g(A)$ and every $n \in \Z$, we have
		\[
		((-x - n) \lor 0) \land 1 = 1 - (((x + (n+1)) \lor 0) \land 1) = \lnot (((x + (n+1)) \lor 0) \land 1).
		\]
		This proves \eqref{i:restriction}.
		
		Since the bijection in \eqref{i:bijection} preserves the inclusion order in both directions, we have \eqref{i:generate}.
	\end{proof}
	
	It is well known that, if we consider a Boolean algebra $\mathbf{B}$ generated by a bounded distributive sublattice $\mathbf{D}$, any element of $B$ can be written as a finite join of elements of the form $x \wedge \neg y$ (or, equivalently, as a finite meet of elements of the form $x \vee \neg y$) for $x,y \in D$. 
	An analogous result holds for positive MV-algebras and MV-algebras. Whereas the result for Boolean algebras can be derived by standard applications of the distributive and De Morgan's laws, the corresponding one for MV-algebras is not straightforward.
	We present it in the next theorem. 
	
	\begin{theorem} \label{t:free-extension-MV}
		Let $\mathbf{A}$ be a positive subreduct of an MV-algebra $\mathbf{B}$, and suppose that $A$ generates the MV-algebra $\mathbf{B}$.
		For every $x \in B$ there are $n,m \in \N$ and $s_1, \dots, s_n, t_1, \dots, t_n, u_1, \dots, u_m, v_1, \dots, v_m \in A$ such that
		\[
		x = \bigoplus_{i = 1}^{n} s_i \odot \lnot t_i = \bigodot_{i = 1}^{m} u_i \oplus \lnot v_i.
		\]
		%and there are $m \in \N$ and $u_1, \dots, u_m, v_1, \dots, v_m \in B$ such that
		%\[
		%    
		%\]
	\end{theorem}
	
	\begin{proof}
		By \cref{l:sub}, the set
		\[
		M \coloneqq \{ x \in G \mid \forall n \in \Z\  ((x - n) \lor 0) \land 1 \in A \}
		\]
		is closed under $+$, $\lor$, $\land$, $0$, $1$ and $-1$ and generates the unital abelian $\ell$-group $\mathbf{G}$.
		Let $z \in B$.
		By \cref{l:lg-generation}, there are $x,y \in M$ such that $z = x - y$.
		For every $n \in \N$, we have
		\begin{align*}
			((-y - n + 1) \lor 0) \land 1 & = 1 - (1 - (((-y - n + 1) \lor 0) \land 1))\\
			& = 1 - (1 + (((y + n - 1) \land 0) \lor -1))\\
			& = 1 - (((y + n) \land 1) \lor 0)\\
			& = 1 - (((y + n) \lor 0) \land 1)\\
			& = \lnot(((y + n) \lor 0) \land 1).
		\end{align*}
		Therefore, by \cref{l:sum-iso}, we have
		\begin{align*}
			z  & = \bigoplus_{n \in \Z} \big((((x - n) \lor 0) \land 1) \odot (((-y - n + 1) \lor 0) \land 1)\big)\\
			& = \bigoplus_{n \in \Z} \big((((x - n) \lor 0) \land 1) \odot \lnot(((y + n) \lor 0) \land 1)\big).
		\end{align*}
		Since $x,y \in M$, $((x - n) \lor 0) \land 1, ((y + n) \lor 0) \land 1 \in A$.
		This proves the first equality in the statement.
		The second one is analogous.
	\end{proof}

	\begin{theorem}\label{t:free-extension-lg}
		Let $M$ be a generating subset of a unital abelian $\ell$-group $\mathbf{G}$, and suppose that $M$ is closed under $+$, $\lor$, $\land$, $0$, $1$ and $-1$.
		For every unital Abelian $\ell$-group $\mathbf{H}$ and every function $f \colon M \to H$ that preserves $+$, $\lor$, $\land$, $0$, $1$ and $-1$, there exists a unique morphism $g \colon \mathbf{G} \to \mathbf{H}$ of unital abelian $\ell$-groups that extends $f$.
		%\[
		%	\begin{tikzcd}
			%    	M \arrow[hook]{r} \arrow[swap]{rd}{\varphi} & G \arrow[dashed]{d}{\exists! \psi}\\
			%    	& H
			%	\end{tikzcd}
		%\]
	\end{theorem}
	
	\begin{proof}
		This follows from \cref{l:lg-generation}. For every $z \in G$ we define the morphism $g$ as $g(z) = f(x) - f(y)$ where $x$ and $y$ are elements of $M$ such that $z=x-y$. The fact that $g$ is a well-defined morphism follows from the fact that $x - y = u - v$ is equivalent to $x + v = u + y$.
	\end{proof}

	The following theorem generalizes an analogous result for bounded distributive lattices, namely that the inclusion of a bounded distributive lattice into a Boolean algebra that is generated by the image of such inclusion is universal, i.e.\ it is a so-called Booleanization, or free Boolean extension: see \cite[Theorem~4.1]{Peremans57} for a version of this result for (not necessarily bounded) distributive lattices.

	\begin{theorem}
		Let $\mathbf{A}$ be a positive MV-algebra, $\mathbf{B}$ an MV-algebra, and $i \colon A \hookrightarrow B$ an injective function that preserves $\oplus$, $\odot$, $\lor$, $\land$, $0$ and $1$ and such that its image generates the MV-algebra $\mathbf{B}$.
		Then $(\mathbf{B}, i)$ is the free MV-extension of $\mathbf{A}$.
	\end{theorem}
	
	\begin{proof}
		Without loss of generality, we may suppose $A \subseteq B$ and $i$ to be the inclusion of $A$ into $B$.
		Let $\mathbf{C}$ be an MV-algebra, and let $f \colon A \to C$ be a function that preserves $\oplus$, $\odot$, $\lor$, $\land$, $0$ and $1$.
		We shall prove that there exists a unique MV-homomorphism $g \colon \mathbf{B} \to \mathbf{C}$ that extends $f$.
		Uniqueness follows from the fact that $A$ generates the MV-algebra $\mathbf{B}$.
		Let us prove existence.
		By the equivalence in \cref{t:Gam-is-equivalence}, we obtain morphisms of {\ulms} $\Xi(i) \colon \Xi(\mathbf{A}) \to \Xi(\mathbf{B})$ and $\Xi(f) \colon \Xi(\mathbf{A}) \to \Xi(\mathbf{C})$.
		By \cref{p:preservation}, $\Xi(i)$ is injective, and we may suppose $\Xi(i)$ to be an inclusion.
		By \cref{t:free-extension-lg}, there exists a unique morphism $g' \colon \Xi(\mathbf{B}) \to \Xi(\mathbf{C})$ of unital abelian $\ell$-groups that extends $\Xi(f)$.
		Then, $g \coloneqq \Gamma(g')$ is an MV-homomorphism that extends $f$.
	\end{proof}

	\section*{Acknowledgments}
	We acknowledge the contribution of L.~M.~Cabrer in shaping the first ideas on positive MV-algebras at the initial stage of this research.
	Marco Abbadini was supported by the Italian Ministry of University and Research through the PRIN project n.\ 20173WKCM5 \emph{Theory and applications of resource sensitive logics}.
	Tom\'a\v{s} Kroupa and Sara Vannucci acknowledge the support by the project 
	Research Center for Informatics
	(CZ.02.1.01/0.0/0.0/16\_019/0000765).

	%%%%%%%%%%%%% BIBLIOGRAPHY %%%%%%%%%%%%%%%%

\end{document}